\documentclass[letterpaper, 10 pt, conference]{ieeeconf}  

\IEEEoverridecommandlockouts                              
\usepackage{graphicx} 
\usepackage{amsmath,amssymb}
\usepackage{comment}
\usepackage{algorithm}
\usepackage{algpseudocode}
\usepackage{xcolor}

\usepackage{csquotes}
\usepackage[utf8]{inputenc}
\usepackage[T1]{fontenc}
\usepackage{dsfont}
\newcommand{\Ind}{\mathds{1}}
\usepackage{mathtools}
\usepackage{hyperref}
\let\labelindent\relax
\usepackage{enumitem}
\usepackage{subcaption}
\newtheorem{theorem}{Theorem}
\newtheorem{lemma}{Lemma}
\newtheorem{assumption}{Assumption}

\newtheorem{definition}{Definition}
\newtheorem{remark}{Remark}

\DeclareMathOperator{\Vol}{Vol}

\title{\LARGE \bf
Data-driven Reachable Set Estimation with Tunable Adversarial and Wasserstein Distributional Guarantees*
}

\author{Georgios Pantazis and Michelle S. Chong
\thanks{*This work is funded through the CETPartnership’s ProRES project under the Joint Call 2024, co‑funded by the European Commission (Grant Agreement No. 101069750) and participating national funding organizations.}
\thanks{G. Pantazis and M. S. Chong are with the Dynamics and Control (D\&C) section at the Department of Mechanical Engineering, Eindhoven University of Technology, the Netherlands. 
Email:{\tt\small \{G.Pantazis33, M.S.T.Chong\}@tue.nl}.}%
}

\begin{document}

\maketitle
\thispagestyle{empty}
\pagestyle{empty}

\begin{abstract}

We study finite horizon reachable set estimation for unknown discrete-time dynamical systems using only sampled state trajectories. Rather than treating scenario optimization as a black-box tool, we show how it can be tailored to reachable set estimation, where one must learn a family of sets based on whole trajectories, while preserving probabilistic guarantees on future trajectory inclusion for the entire horizon. To this end, we formulate a relaxed scenario program with slack variables that yields a tunable trade-off between reachable set size and out-of-sample trajectory inclusion over the horizon, thereby reducing sensitivity to outliers. Leveraging the recent results in adversarially robust scenario optimization, we then extend this formulation to account for bounded adversarial perturbations of the observed trajectories and derive a posteriori probabilistic guarantees on future trajectory inclusion. When probability distribution shifts  in the Wasserstein distance occur, we obtain an explicit bound on how gracefully the theoretical probabilistic guarantees degrade. For different geometries, i.e., $p$-norm balls, ellipsoids, and zonotopes, we derive tractable convex reformulations and corroborate our theoretical results in simulation.
\end{abstract}


\section{Introduction}
\subsection{Motivation}
Reachable sets describe the potential states a dynamical system can attain over a finite horizon under admissible uncertainty, initial conditions and inputs, and form the foundation for verification, safety analysis, and constrained control \cite{Althoff2021}. While reachable set estimation is mature for well-modeled systems \cite{Chong2022_ellipsoidal}, many modern applications rely on partially known dynamics or are affected by uncertainty sources difficult to model without additional information. At the same time, trajectory data are increasingly abundant, collected from simulation or experiments. These trends have motivated data-driven reachability methods that infer reachable sets directly from sampled trajectories while retaining guarantees that extrapolate beyond the observed trajectories \cite{Alanwar2023, wang2023svdd}.

Recent work on reachable state estimation spans several complementary approaches. One approach estimates reachable sets directly from data using functional or geometric set estimators, including methods based on Christoffel functions with finite-sample guarantees \cite{devonport2023christoffel} and support vector data description for Lipschitz nonlinear systems \cite{wang2023svdd}. Another approach formulates reachable set estimation as a scenario program \cite{Calafiore2006}, thereby leveraging randomized optimization to learn reachable sets with probabilistic guarantees for unseen trajectories \cite{devonport20a,dietrich2024nonconvex}.  \cite{tebjou2023} studies direct data-driven reachable-set approximation using Christoffel functions and conformal prediction, yielding finite-sample coverage guarantees. The work in \cite{lin2024} studies verification of candidate neural reachable tubes via scenario optimization and conformal prediction.  More recently, the work in \cite{Dietrich2025} studies data-driven reachability through the holdout method, deriving \emph{a posteriori} guarantees for reachable sets and reachable tubes and comparing them with wait-and-judge bounds \cite{Campi2018wait}. In parallel, probabilistic reachability has been studied through stochastic models, including Gaussian processes \cite{griffioen2023gpreach} and learning-based models of human-in-the-loop systems \cite{govindarajan2017data,choi2023data}. More recently, robustness to training--deployment mismatch has been addressed through surrogate modeling and conformal prediction under distribution shift based on KL-divergence \cite{hashemi2024statistical}. 

 Despite this progress, three practical challenges are still rarely addressed simultaneously in reachable set estimation: 1) A tunable trade-off between set tightness and reliability, 2) robustness to perturbed trajectory data, and 3) guarantees that degrade gracefully under distribution shifts based on the Wasserstein distance. This paper addresses these challenges in a unified framework with the scenario approach as its backbone.

\subsection{Main Contributions}

Specifically, our contributions with respect to the related literature are as follows: 
\begin{enumerate}
\item Prior works that leverage scenario optimization for reachable set estimation \cite{devonport20a, dietrich2024nonconvex} rely on scenario programs with hard constraints for the inclusion of observed trajectories. \cite{lin2024} focuses on verification of candidate neural reachable tubes based on the sampling-and-discarding approach \cite{campi2011sampling}, while  \cite{Dietrich2025} develops probabilistic guarantees for reachable sets based on the holdout method, comparing the results with the bounds in wait-and-judge scenario optimization \cite{Campi2018wait}. This paper instead leverages the scenario approach with relaxation \cite{CampiGaratti2021RiskRelaxation} to develop a reachable set estimation framework with auxiliary decision variables that relax selected trajectories. In this way, it bridges \cite{CampiGaratti2021RiskRelaxation} with data-driven reachable set estimation and yields a tunable trade-off between reachable set size and the theoretical guarantees for inclusion of future trajectories over the considered horizon.  
\item Existing reachable set estimation methods, based on scenario optimization, treat the sampled trajectories as absolute truth and provide guarantees only for the nominal distribution \cite{devonport20a,dietrich2024nonconvex,lin2024,Dietrich2025}. Recent robust reachability methods based on robust conformal prediction rely on surrogate modeling and calibration \cite{hashemi2024statistical} to address this. 
We instead leverage results from adversarially robust scenario optimization \cite{Campi_RiskAnalysis_2025} to propose a method for reachable set estimation under  perturbed sampled trajectories. Specifically, we model each sampled trajectory as subject to bounded adversarial perturbations, known or estimated, and derive \emph{a posteriori} bounds on the probability of future trajectory inclusion by solving a robustified version of our problem.  If the probability distribution of the state trajectories shifts upon deployment, but lies in a 1-Wasserstein ball around the nominal  distribution, we obtain an explicit out-of-distribution bound. 
\item Different geometric shapes of candidate reachable sets yield different reformulations. We thus perform a parametric analysis of the size vs. future trajectory inclusion trade-off for $p$-normed balls, ellipsoids and zonotopes. The obtained reformulations are convex and amenable to efficient numerical computation.
\end{enumerate}
 The rest of the paper is organized as follows: In Section II we formulate the problem of learning a reachable set with tunable trade-off between size and future trajectory inclusion. In Section III we develop the framework on adversarially robust reachable set estimation with probabilistic guarantees. In Section IV we derive tractable convex formulations for different set geometries such as $p$-normed balls, ellipsoids, and zonotopes. In Section V, we focus on the distributional robustness of the estimated reachable sets in the  Wasserstein distance. Section VI corroborates our results via simulation studies, Section VII concludes the paper, while Section VIII is the Appendix.
 
 \emph{Notation:}
We consider the index set $\mathcal{T}=\{0,\ldots,T\}$, where $T$ denotes the horizon length, and $ \mathcal{N}=\{1,\ldots,N\}$ is the index set of sample trajectories with $N$ being the number of samples.  For a collection of states, we use the equivalent notation $(x_0^\top, \dots, x^\top_T)^\top=(x_k)_{k \in \mathcal{T}}$, where $x_k \in \mathbb{R}^{n_x}$ is a column vector. We denote the cartesian product among $L$ sets $Y_i$ as $\prod_{i=0}^{L-1} Y_i$. $\mathbb{P}^N = \prod_{i \in \mathcal{N}} \mathbb{P}$ denotes the cartesian product of $N$ copies of the same probability distribution $\mathbb{P}$. $\Vol(B)$ denotes the volume of a set $B$. $A^\top$ denotes the transpose of a matrix $A$, while $\det(A)$ denotes its determinant. $\mathbb{S}^m_{\succ 0}$ denotes the set of positive definite symmetric matrices of dimension $m \times m$. Let $\mathbb{B}_p^{n_x}$ be the $p$-normed ball of dimension $n_x$, and $\|y\|_p$ the $p$-norm of a vector $y$. $\mathbb{R}_+$ is the set of nonnegative reals, and $\mathbf{1}_{n}$ the vector of ones of dimension $n$. The indicator function $\Ind[C]$ takes the value 1 when condition $C$ is true and zero otherwise.

\section{Problem Formulation}
We consider a system with (possibly) unknown dynamics:
\begin{align}
&x_{k+1} = f(x_k), \label{state_dynamics}
\end{align}
where $x_k \in X_k \subseteq \mathbb{R}^{n_x}$ denotes the state of the system and $X_k$ denotes the set in which the state $x_k$ takes values. The function $f$ is a continuous function that maps a state $x_k$ to a unique next state $x_{k+1}$. The concatenation of these states defines a trajectory $x=(x_k)_{k \in \mathcal{T}}$. In (\ref{state_dynamics}), the function $f$ is considered to be unknown and incorporates uncertainty in the state dynamics such as model mismatch or disturbances. 
If the dynamics $f$ are explicitly affected by an uncertain disturbance, a special case of the random operator $f$ is obtained, i.e.,  $f(x_k)=\tilde{f}(x_k, w_k)$, where $w_k \in \mathbb{R}^{n_w}$ is an uncertain disturbance parameter. Note that no assumption on the boundedness of the support set of $w_k$ for each timestep $k \in \mathcal{T}$ is considered, thus allowing for possibly unbounded disturbances.  

Since we have no knowledge of the dynamics, we assume access to $N$ different state trajectories $x^{(i)}=(x^{(i)}_k)_{k \in \mathcal{T}}$, $i \in \mathcal{N}$ obtained from simulations, synthetic data and/or past experiments. For the case  $f(x_k)=\tilde{f}(x_k, w_k)$, these samples represent the trajectories due to different realizations of the initial condition $x_0$ and the disturbance $w_k$, $k \in \mathcal{T}$. The state trajectory is thus modeled as a random vector in $X = \prod_{k \in \mathcal{T}} X_k\subset \mathbb{R}^{(T+1)n_x}$ that satisfies \eqref{state_dynamics} and is drawn from a (possibly) unknown probability distribution $\mathbb{P}$.  We impose the following standing assumption: 
\begin{assumption} \label{assum:iid}

 We have access to $N$ independent sample  trajectories $x^{(i)}$ from the probability distribution $\mathbb{P}$.
 \hfill $\square$
\end{assumption}
 
Note that, in Assumption \ref{assum:iid}, the probability distribution $\mathbb{P}$ is used to describe the entire trajectory vector. As such correlations among the elements of the state trajectories as a result of the dynamics in \eqref{state_dynamics} are allowed.

 \subsection{ Tunable Reachable Set Estimation: Size vs Robustness}
We consider a parametric family of sets to be used to estimate the reachable region for the states $x_k$ at time $k \in \mathcal{T}$ defined by
$\mathcal{R}_k( \theta_k) = \{z_k \in \mathbb{R}^{n_x} : g_k(z_k,\theta_k) \le 0\}$, where $\theta_k \in \Theta_k \subseteq \mathbb{R}^{n_\theta}$, and $g_k: \mathbb{R}^{n_x} \times \mathbb{R}^{n_\theta} \rightarrow \mathbb{R}^q$ is a set of constraints that enforce consistency with the specific choice of reachable set geometry. Depending on this choice, $g_k$ takes different forms and the decisions $\theta_k$ have a different geometric meaning. 
\begin{remark} \label{rem:example}
For example, considering the family of $p$-norm balls for reachable sets, we have that $g_k(z_k, \theta_k)=\|z_k-c_k\|_p-r_k$, where $\theta_k=(c_k,r_k) \in \mathbb{R}^{n_x+1}$ with $c_k$ being the center of the reachable set and $r_k$ its radius. \hfill $\square$
\end{remark} 
We propose the following optimization program over a horizon of length $T$:
\begin{align}
\text{$P_{ N, \rho}$ :}\quad
\min_{\theta \in \Theta, \xi \in \mathbb{R}_{\geq 0}^N}\ & \sum_{k \in \mathcal{T}} \left[ S\bigl(\mathcal{R}_k(\theta_k)\bigr) + \rho \sum_{i \in \mathcal{N}} \xi_{i} \right] \nonumber\\
\text{s.t.}\ & g\bigl(x^{(i)} ,\theta \bigr)\le \xi_{i},\quad \forall i\in \mathcal{N}, \nonumber
\end{align}
where $\theta=(\theta_k)_{k \in \mathcal{T}}$ is a vector of concatenated decision variables for the candidate reachable set, $\xi=(\xi_i)_{i \in \mathcal{N}}$ the vector of relaxation variables and $g$ is obtained by taking the maximum over the horizon, i.e., $g(x, \theta)=\max_{k \in \mathcal{T}} g_k(x_k, \theta_k)$. $S(\mathcal{R}_k(\theta_k))$ denotes an appropriate size proxy of the set $\mathcal{R}_k(\theta_k)$, e.g., its volume $\Vol(\mathcal{R}_k(\theta_k))$. An alternative size proxy  can be the radius for the $p$-normed ball or the scaling factors of an ellipsoid or zonotope (see Section \ref{sec:tractable}). To ensure that our approach is computationally feasible, we will focus on convex reachable sets and their corresponding reformulations.  The optimal solution of program \text{$P_{ N, \rho}$} is denoted by $\theta^\ast_{N, \rho}$ and feeding this solution to the vector-to-set mapping $\mathcal{R}_k(\theta_k)$  produces the estimated reachable sets for each time step $k \in \mathcal{T}$, denoted by $\mathcal{R}^\ast_{k,N}= \mathcal{R}_{k}(\theta^\ast_{k, N, \rho})$. We denote the collection of these sets by $\mathcal{R}^\ast_N=\{\mathcal{R}^\ast_{k,N}\}_{k \in \mathcal{T}}$.
\begin{remark}
Taking the limit $\rho \to \infty$, problem $P_{N,\rho}$ recovers the reachable set formulation in \cite{devonport20a}. Its main advantage is the introduction of relaxation variables $\xi_i$ for each sampled trajectory, which provide additional flexibility. In particular, varying $\rho$ allows one to trade off the size of the estimated reachable set against consistency with the collected state trajectories. Larger values of $\rho$ place greater emphasis on trajectory consistency and less on the minimization of reachable set size. As shown in the subsequent developments, this trade-off also affects the probabilistic guarantees for inclusion of future state trajectories over the horizon. \hfill $\square$
\end{remark}

 Having obtained our optimal collection of reachable sets $\mathcal{R}^\ast_k$ for each $k \in \mathcal{T}$ by solving $P_{N, \rho}$, we now wish to evaluate how well they generalize to new yet unseen state trajectories. Thus, we define the \emph{probability of trajectory exclusion}. 
\begin{definition} \label{prob_viol}
Let  $\mathcal{R}^\ast_{k,N}=\mathcal{R}_{k}(\theta^\ast_{k, N, \rho})$ be the estimated reachable set at time $k \in \mathcal{T}$ obtained from $P_{N, \rho}$.  Then, the probability of trajectory exclusion is defined as:
\begin{align}
\mathbb{V}(\mathcal{R}^\ast_N) := \mathbb{P}\Bigl\{x \in X: \exists \ k \in \mathcal{T}  \text{ such that } x_k \notin \mathcal{R}^\ast_{k,N}\Bigr\}, \nonumber
\end{align}
where $x=(x_k)_{k \in \mathcal{T}}$ is a new yet unseen state trajectory. \hfill  $\square$ 
\end{definition}

The terms trajectory exclusion and trajectory violation are used interchangeably. In words, $\mathbb{V}(\mathcal{R}^\ast_N)$ is the probability that a new trajectory is drawn from $\mathbb{P}$ according to the dynamics in (\ref{state_dynamics}), such that for some timestep $k \in \mathcal{T}$, the state $x_k$ is not contained in the estimated set $\mathcal{R}^\ast_{k,N}$.
Though general, this violation metric does not account for state trajectories, perturbed due to noise or poisoning from adversaries, which can occur upon deployment.

\section{Adversarially Robust Reachable Sets}

 We consider that each sampled state trajectory can be perturbed, due to noise or the presence of an adversary, from its nominal value by a radius $R$ according to a well-defined distance metric. The radius $R$ is considered to be either a designer's choice or an approximation based on adversarial observations, and depicts the trust the designer puts in the sampled trajectory data. As such, the set of adversarial perturbations around a nominal sample $x$ is given by $A_{x}:=\{\tilde{x}: d(x,\tilde{x}) \leq R\},$
where  $d: \mathbb{R}^n \times \mathbb{R}^n \rightarrow \mathbb{R}_{\geq 0}$ is a distance metric such as the $p$-normed difference $d(x, \tilde{x})= \|x-\tilde{x}\|_p$. We thus propose an extension of our reachable set estimation problem $P_{N, \rho}$ that accounts for trajectory perturbations upon deployment: 
\begin{align}
\text{$P_{ N, \rho, A}$ :}\quad
&\min_{\theta \in \Theta, \xi \in \mathbb{R}_{\geq 0}^N}\  \sum_{k \in \mathcal{T}} \left[ S\bigl(\mathcal{R}_k(\theta_k)\bigr) + \rho \sum_{i \in \mathcal{N}} \xi_{i} \right] \nonumber\\
& \text{s.t. } g\bigl(\tilde{x}^{(i)},\theta\bigr)\le \xi_{i},\quad  \forall \tilde{x}^{(i)} \in A_{x^{(i)}}, \ \forall i\in \mathcal{N}. \nonumber
\end{align}
In our setting, we assume that either $A_x$ is known or  an approximation $\hat{A}_x$ can be obtained. More information with regard to methods used to construct such adversarial set approximations can be found in \cite{Campi_RiskAnalysis_2025}. Hence, program $\hat{P}_{ N, \rho, A}$ can be approximated by the program: 
\begin{align}
\text{$\hat{P}_{ N, \rho, A}$ :}\quad
&\min_{\theta \in \Theta, \xi \in \mathbb{R}_{\geq 0}^N}\  \sum_{k \in \mathcal{T}} \left[ S\bigl(\mathcal{R}_k(\theta_k)\bigr) + \rho \sum_{i \in \mathcal{N}} \xi_{i} \right] \nonumber\\
& \text{s.t. } g\bigl(x^{(i,j)},\theta \bigr)\le \xi_{i},\quad  \forall j \in \mathcal{M}_i, \ \forall i\in \mathcal{N}. \nonumber
\end{align}
where $\mathcal{M}_i$ is the set of vertices of a polytopic approximation of the adversarial set $\hat{A}_{x^{(i)}}$.  In case the adversarial set is a known polytope, the approximation is exact. 
Note that due to the auxiliary variables $\xi_i$, an optimizer of $\hat{P}_{ N, \rho, A}$ always exists. Depending on the geometric shapes selected for $\mathcal{R}_k(\theta_k), k \in \mathcal{T}$, which results in different constraints $g(x, \theta)$ (see Remark \ref{rem:example}), solving $\hat{P}_{ N, \rho, A}$ can be computationally challenging. To this end, we impose the following assumption.
\begin{assumption} \label{assum:convex_vol}
For each $k \in \mathcal{T}$, the size proxy $S\!\left(R_k(\theta_k)\right)$ is convex in $\theta_k$. Moreover, for each fixed $x$, the constraint function $g(x,\theta)$ is convex in $\theta$. \hfill $\square$
\end{assumption}
Assumption \ref{assum:convex_vol} ensures that the scenario program $\hat{P}_{N,\rho, A}$ is convex.  However, convexity of the cost and constraints alone does not guarantee uniqueness of the optimizer (see  \cite[Table 1]{mestres2025regularity}). Hence, we assume the following:

\begin{assumption} \label{assum:unique_opt}
Problem $\hat{P}_{ N, \rho, A}$ admits a unique optimal solution $\theta_{N, \rho,A}^\ast$. \hfill $\square$
\end{assumption}

Assumption \ref{assum:unique_opt} implies that a unique reachable set $\mathcal{R}^\ast_{k,N,\rho, A}$ is returned for each time step $k \in \mathcal{T}$. In case multiple solutions exist, the use of a convex tie-break rule, such as the minimum  norm or lexicographic order, can be used to extract a unique solution. Based on this setting, we introduce the concept of \emph{adversarial probability of trajectory exclusion}: 

\begin{definition} \label{prob_viol_adv}
Let  $\mathcal{R}^\ast_{k,N,\rho, A}$ be the unique estimated reachable set at time $k \in \mathcal{T}$ obtained from $\hat{P}_{ N, \rho, A}$ and $\mathcal{R}^\ast_N=\{\mathcal{R}^\ast_{k,N,\rho, A}\}_{k \in \mathcal{T}}$.  Then, the adversarial probability of trajectory exclusion is defined as:
\begin{align}
\mathbb{V}_A(\mathcal{R}^\ast_N) := \mathbb{P}\Bigl\{x : \exists \ \tilde{x} \in A_{x},  
\exists k \in \mathcal{T}  \text{ s.t. } \tilde{x}_k \notin \mathcal{R}^\ast_{k,N,\rho, A}\Bigr\}.   \nonumber
\end{align}
\end{definition}

Note that the collection of reachable sets obtained by solving $\hat{P}_{ N, \rho, A}$ can be written as $\mathcal{R}^\ast_{k, N, \rho, A}= \{x_k \in \mathbb{R}^{n_x}: g_k(x_k, \theta^\ast_{k,N,\rho, A}) \leq 0\}$. We now introduce the notion of adversarial complexity appropriately adapted from \cite{Campi_RiskAnalysis_2025} into our reachable set estimation framework:

\begin{definition} \label{def:adv_com}
Given the collection of state trajectories $\{x^{(i)}\}_{i \in \mathcal{N}}$, define the adversarial complexity $s^\ast_{A}$ as the number of indices
$i\in \mathcal{N}$ such that at least one of the following holds:
\begin{enumerate}
\item[(i)] There exists \( \tilde x\in \hat{A}_{x^{(i)}}\) such that for some $k \in \mathcal{T}$ it holds that \(g_k(\tilde x_k,\theta^\ast_{k, N, \rho, A})>0\);
\item[(ii)] There exists \( \tilde{x}\in \hat{A}_{x^{(i)}}\) such that for some $k \in \mathcal{T}$\ it holds that $g_k(\tilde x_k,\theta^\ast_{k,N, \rho, A})=0$;
\item[(iii)] There exists \( \tilde x \in A_{x^{(i)}}\) such that for some $k \in \mathcal{T}$ it holds that \(g_k(\tilde x_k,\theta^\ast_{k,N, \rho, A})>0\).
\end{enumerate}
\end{definition}
 Definition \ref{def:adv_com} follows by the observation that $\{ x: \exists k \in \mathcal{T} \text{ such that } x_k \notin \mathcal{R}^\ast_{k, N, \rho, A}\}= \{ x: \exists k \in \mathcal{T} \text{ such that } g_k(x_k, \theta_k) > 0\}$ and then invoking Definition 7 in \cite{Campi_RiskAnalysis_2025}. The adversarial complexity $s^\ast_A$ denotes the amount of state trajectories which are important for the construction of the reachable sets across the horizon. In general, a relatively low value of $s^\ast_A$ for a fixed $N$ implies a better ability of the reachable sets to include yet unseen trajectories (see Theorem \ref{thm:1}). Furthermore, we impose the following non-accumulation assumption. 

\begin{assumption} \label{assum:nonaccumulation}
For every decision \(\theta\in\Theta \), it holds that:
\[
\mathbb P\Big\{x\in  X:\ \exists \ \tilde x \in \hat{A}_x \text{ such that } g(\tilde x,\theta)=0\Big\}=0. \hfill \ \  \square
\] 
\end{assumption}  
 Assumption \ref{assum:nonaccumulation}   prevents degeneracies in the scenario problem and is a standard assumption for uncertainties following a continuous probability distribution, where exact equalities occur with probability zero. We now define the  equations for the  violation levels to be used for our probabilistic result in Theorem \ref{thm:1}, based on \cite{Campi_RiskAnalysis_2025} and \cite{CampiGaratti2021RiskRelaxation}.
\begin{definition} \label{def:epsilon_equations}
For each $\nu \in\{0,1,\dots,N-1\}$, consider the polynomial equation in the variable $t\ge 0$:
\begin{equation}\label{eq:eps_poly_nu}
\binom{N}{\nu} t^{\,N-\nu}
-\frac{\beta}{2N}\sum_{i=\nu}^{N-1}\binom{i}{\nu} t^{\,i-\nu}
-\frac{\beta}{6N}\sum_{i=N+1}^{4N}\binom{i}{\nu} t^{\,i-\nu}
=0 .
\end{equation}
Let its two nonnegative roots be $\underline{t}(\nu)\le \overline{t}(\nu)$.
For $\nu=N$, consider
\begin{equation}\label{eq:eps_poly_N_nu}
1-\frac{\beta}{6N}\sum_{i=N+1}^{4N}\binom{i}{N} t^{\,i-N}=0,
\end{equation}
which has a unique nonnegative root $\underline{t}(N)$, and set $\overline{t}(N):=0$.
Define, for all $\nu \in \mathcal{N}$,
$\overline{\epsilon}(\nu):=1-\underline{t}(\nu)$ and
$\underline{\epsilon}(\nu):=\max\{0,\;1-\overline{t}(\nu)\}.$
\end{definition}

Equations (\ref{eq:eps_poly_nu}) and (\ref{eq:eps_poly_N_nu}) analyzed in \cite{CampiGaratti2021RiskRelaxation}, \cite{Campi_RiskAnalysis_2025} establish a connection among the number of samples $N$, the complexity $s^\ast_A$ (written as $\nu$), and the confidence parameter $\beta$, and determine the theoretical violation levels $\underline{\epsilon}$ and $\overline{\epsilon}$. To solve (\ref{eq:eps_poly_nu}) and (\ref{eq:eps_poly_N_nu}) efficiently, a bisection algorithm can be used  \cite[Appendix]{CampiGaratti2021RiskRelaxation}. The following result then holds:

\begin{theorem}  \label{thm:1}
Consider system (\ref{state_dynamics}), its estimated reachable set $\mathcal{R}^\ast_{k,N,\rho,A} := \mathcal{R}_k(\theta^\ast_{k,N,\rho,A})$  
with $\mathcal{R}^\ast_N := \{\mathcal{R}^\ast_{k,N,\rho,A}\}_{k \in \mathcal{T}}$ obtained via $\hat{P}_{N, \rho, A}$ under Assumptions~1-4. Let $s_A^\ast$ be the adversarial complexity from  Definition \ref{def:adv_com}.
Then, for any confidence parameter $\beta\in(0,1)$, it holds that:
\[
\mathbb{P}^N\!\left\{\,\{x^{(i)}\}_{i \in \mathcal{N}} :
\underline{\epsilon}\!\left(s_A^\ast\right)\le \mathbb{V}_A(\mathcal{R}^\ast_N)\le \overline{\epsilon}\!\left(s_A^\ast\right)\right\}\ge 1-\beta,
\]
where the
violation levels $\underline{\epsilon}(\cdot)$ and $\overline{\epsilon}(\cdot)$ are
obtained from (\ref{eq:eps_poly_nu}) and (\ref{eq:eps_poly_N_nu}). \hfill $\square$
\end{theorem}

\emph{Proof}: See Appendix. \hfill $\blacksquare$

 Theorem \ref{thm:1} states that for a number of samples $N$ and a parameter $\beta \in (0,1)$, with confidence at least $1-\beta$, the adversarial probability of trajectory exclusion $\mathbb{V}_A(\mathcal{R}^\ast_N)$ lies between $\underline{\epsilon}(s^\ast_A)$ and $\overline{\epsilon}(s^\ast_A)$, bounds that exhibit a monotonically increasing relation with respect to the adversarial complexity $s^\ast_A$.

\section{Tractable Reformulations} \label{sec:tractable}
In order, for Theorem \ref{thm:1} to be useful, problem $\hat{P}_{N, \rho, A}$ should be tractable. This section derives tractable reformulations of $\hat{P}_{N, \rho, A}$ for different choices of geometric shapes, commonly used in reachable set estimation. 
 We impose one of the following assumptions with regard to the structure of the reachable set:

\begin{assumption} \label{assum:balls}
For each $k \in \mathcal{T}$, the reachable sets are $p$-normed balls, i.e., $\mathcal{R}_k(c_k,r_k) := \{x\in\mathbb{R}^{n_x} : \|x-c_k\|_p \le r_k\}$
with radius $r_k \in \mathbb{R}_{\geq 0}$, and centre $c_k\in\mathbb{R}^{n_x}$.
\hfill $\square$
\end{assumption}

\begin{assumption} \label{assum:ellip}
   For each \(k \in \mathcal{T}\), the reachable sets are ellipsoids, i.e., 
$\mathcal{R}_k(C_k,b_k):=\{x \in \mathbb{R}^{n_x}: \|C_kx+b_k\|_2 \leq 1\}$,
where $C_k \in \mathbb{S}^{n_x}_{\succ 0}$ and $b_k \in \mathbb{R}^{n_x}$. \hfill $\square$
\end{assumption}
\begin{assumption} \label{assum:zono}
 For each $k \in \mathcal{T}$, the reachable sets are zonotopes, i.e., 
$\mathcal{R}_k(c_k,a_k)
:= \Big\{x \in \mathbb{R}^{n_x} : \exists \zeta \in \mathbb{R}^m,\;
x = c_k + G_k \zeta,\; |\zeta| \le a_k \Big\}$,
where $G_k \in \mathbb{R}^{n_x \times m}$ is a generator matrix, $\zeta \in \mathbb{R}^m$ a scaling factor, $c_k \in \mathbb{R}^{n_x}$ the centre and $a_k \in \mathbb{R}^m_{+}$ the half-widths. The inequality $ |\zeta| \le a_k$ is understood component-wise.  \hfill $\square$
\end{assumption}

\subsection{Reformulations based on volume}
Denote the set of indices $\mathcal{K} := \{ (j,k,i) : i \in \mathcal{N},\; k \in \mathcal{T},\; j \in \mathcal{M}_i \}$. Then, the following exact  reformulations are obtained.

\begin{lemma} \label{lem:reformulations}
\begin{enumerate}[label=\roman*)]
\item Under Assumption \ref{assum:balls}, consider $S(\mathcal{R}_k(c_k, r_k))=\Vol(\mathcal{R}_k(c_k, r_k))$. Furthermore,  denote $c=(c_k)_{k \in \mathcal{T}}$, $r=(r_k)_{k \in \mathcal{T}}$, $\xi=(\xi_i)_{i \in \mathcal{N}}$. Then problem $\hat{P}_{ N, \rho, A}$ can be written as:
\begin{equation*}
\left\{
\begin{aligned}
&\min_{
(c, r, \xi) \in \mathbb{R}^{n_x T} \times  \mathbb{R}^T_{\ge 0} \times  \mathbb{R}^N_{\ge 0}
}
\sum_{k \in \mathcal{T}}
\left(
\mathrm{Vol}(\mathbb{B}^{n_x}_p) r_k^{n_x}
+\rho \sum_{i \in \mathcal{N}} \xi_i
\right)
\\
&\mathrm{s.t.}
\left\|x_k^{(i,j)}-c_k\right\|_p \le r_k+\xi_i,
\forall (j ,k, i)\in \mathcal{K}.
\end{aligned}
\right.
\end{equation*}
\item Under Assumption \ref{assum:ellip}, consider $S(\mathcal{R}_k(C_k, b_k))= \Vol(\mathcal{R}_k(C_k, b_k))$ and denote $C=(C_0; \dots; C_T)$, $b=(b_{k})_{k \in \mathcal{T}}$ and $\xi=(\xi_i)_{i \in \mathcal{N}}$. Then,  $\hat{P}_{N, \rho, A}$ can be written as: 
\begin{equation*}
\left\{
\begin{aligned}
&\min_{\substack{
(C, b) \in \prod\limits_{k \in \mathcal{T}}\mathbb{S}^{n_x}_{ \succ 0} \times \mathbb{R}^{n_x},  \\
\xi \in  \mathbb{R}^N_{\ge 0}
}}
\sum_{k \in \mathcal{T}}
\left(
\frac{\mathrm{Vol}(\mathbb{B}_2^{n_x})}{\det(C_k)}
+\rho \sum_{i \in \mathcal{N}} \xi_i
\right)
\\[0.4em]
&\mathrm{s.t.}
\left\|C_k x_k^{(i,j)} + b_k\right\|_2 \le 1+\xi_i,
\forall (j ,k, i)\in \mathcal{K}. 
\end{aligned}
\right.
\end{equation*}
\end{enumerate}

\end{lemma}

\emph{Proof:} See Appendix.

The exact reformulations in Lemma \ref{lem:reformulations} can still be challenging  to solve in general. To this end, alternative proxies of size are often used to obtain a computationally tractable reachable region. Note that such regions also enjoy the theoretical guarantees of Theorem \ref{thm:1}, though with possibly different adversarial complexity. For example, under Assumption \ref{assum:ellip}, consider the monotone proxy $S(\mathcal{R}_k)=\log(\text{Vol}(\mathcal{R}_k))$ in problem  $\hat{P}_{ N, \rho, A}$ for each $k \in \mathcal{T}$. Then $\hat{P}_{ N, \rho, A}$   can be written as: 
\begin{equation*}
\left\{
\begin{aligned}
&\min_{\substack{
(C, b) \in \prod\limits_{k \in \mathcal{T}}\mathbb{S}^{n_x}_{ \succ 0} \times \mathbb{R}^{n_x}  \\
\xi \in  \mathbb{R}^N_{\ge 0}
}}
\sum_{k \in \mathcal{T}}
\left(
-\log\!\bigl(\det(C_k)\bigr)
+\rho \sum_{i \in \mathcal{N}} \xi_i
\right) \nonumber 
\\  
&\text{s.t.} \left\|C_k x_k^{(i,j)} + b_k\right\|_2 \le 1+\xi_i, 
\forall (j ,k, i)\in \mathcal{K}.
\end{aligned}
\right.
\end{equation*}
The reformulation is based on the fact that $C_k$ is positive definite, which implies that $\det(C_k) > 0$. The logarithm of volume of an ellipsoid then takes the form $\log (\text{Vol}(\mathcal{R}_k(C_k,b_k)))= \log(\text{Vol}(\mathbb{B}^{n_x}_2)/ \det(C_k))=\log(\text{Vol}(\mathbb{B}^{n_x}_2))-\log(\det(C_k))$. Since $\text{Vol}(\mathbb{B}^{n_x}_2)$ is a constant, the optimizer of the problem at hand is not affected by the translation this term imposes. 

\subsection{Reformulations based on  size proxies}

The corresponding proxy reformulations of some common shapes, used in reachable set estimation, are established in the following lemma:

\begin{lemma} \label{lem:2}
\begin{enumerate}[label=\roman*)]
\item Under Assumption \ref{assum:balls},  consider the proxy $S(\mathcal{R}_k)=r_k$ and denote $c=(c_k)_{k \in \mathcal{T}}$, $r=(r_k)_{k \in \mathcal{T}}$, $\xi=(\xi_i)_{i \in \mathcal{N}}$. Then, problem $\hat{P}_{ N, \rho, A}$ can be written as:
\begin{equation*}
\left\{
\begin{aligned}
&\min_{\substack{
(c, r, \xi) \in \mathbb{R}^{n_xT} \times  \mathbb{R}^T_{\ge 0} \times \mathbb{R}^N_{\ge 0}
}}
\sum_{k \in \mathcal{T}}
\left(
r_k+\rho \sum_{i \in \mathcal{N}} \xi_i
\right)
\\ 
& \mathrm{s.t.}
\left\|x_k^{(i,j)}-c_k\right\|_p \le r_k+\xi_i,
\forall (j ,k, i)\in \mathcal{K}. 
\end{aligned}
\right.
\end{equation*}
\item For ellipsoidal sets of the form $\mathcal{R}_k(c_k,s_k)=\{x: \|H_k (x-c_k)\|_2 \leq s_k\}$, where $H_k \in \mathbb{S}^{n_x}_{\succ 0}$ is fixed for each $k \in \mathcal{T}$, we consider the proxy $S(\mathcal{R}_k)=s_k$ and denote $c=(c_k)_{k \in \mathcal{T}}$, $s=(s_k)_{k \in \mathcal{T}}$ and $\xi=(\xi_{i})_{i \in \mathcal{N}}$. Then, $\hat{P}_{ N, \rho, A}$ takes the form:
\begin{equation}\label{eq:ellipsoid-socp}
\left\{
\begin{aligned}
&\min_{\substack{
\{(c, s ,\xi) \in \mathbb{R}^{n_xT} \times  \mathbb{R}^T_{\ge 0} \times \mathbb{R}^N_{\ge 0}
}}
\sum_{k \in \mathcal{T}}
\left(
s_k + \rho \sum_{i \in \mathcal{N}} \xi_i
\right) \nonumber 
\\
& \mathrm{s.t.}
\left\|H_k\!\left(x_k^{(i,j)}-c_k\right)\right\|_2 \le s_k+\xi_i,
\forall (j ,k, i)\in \mathcal{K}. 
\end{aligned}
\right.
\end{equation}
\item Under Assumption \ref{assum:zono}, consider $S(\mathcal{R}_k)= \mathbf{1}_m^\top a_k$ with $G_k \in \mathbb{R}^{n_x \times m}$ fixed, and denote $c=(c_k)_{k \in \mathcal{T}}$, $a=(a_k)_{k \in \mathcal{T}}$. For each $(j,k,i)\in\mathcal K$, let $\zeta_{i,k,j}
=
\begin{bmatrix}
\zeta_{i,k,j,1} & \dots & \zeta_{i,k,j,m}
\end{bmatrix}
\in \mathbb{R}^m$ and consider 
$\zeta := \bigl(\zeta_{i,k,j}\bigr)_{(j,k,i)\in\mathcal K}
$. Then, $\hat{P}_{N, \rho, A}$ takes the form:
\begin{equation}
\label{eq:zonotope-lp}
\left\{
\begin{aligned}
&\min_{ (c, a, \xi,  \zeta) \in \mathbb{R}^{n_xT} \times \mathbb{R}^{mT+N}_{\geq 0} \times \mathbb{R}^{m|\mathcal{M}_i|TN} 
}
\sum_{k\in\mathcal T} \mathbf{1}_m^\top a_k \;+\; \rho \sum_{i\in\mathcal N} \xi_i \nonumber
\\[0.4em]
&\mathrm{s.t.} \  c_k + G_k \zeta_{i,k,j} = x_k^{(i,j)},
\forall (j ,k, i)\in \mathcal{K}, \nonumber \\
& \  -(a_k + \xi_i \mathbf{1}_m)
\le
\zeta_{i,k,j}
\le
a_k + \xi_i \mathbf{1}_m, \forall (j,k,i)\in\mathcal{K}. 
\end{aligned}
\right.
\end{equation}
\end{enumerate}
\end{lemma}

\emph{Proof}: See Appendix.
 
\begin{remark}
Note that the penalty parameter $\rho$ should be tuned in accordance with the proxy and geometry used for reachable set estimation, as they directly affect the associated scenario problem. Section VI.A supports this observation by showing that, for zonotopes, larger values of $\rho$ are often required to enforce consistency with the sampled trajectories than for balls or ellipsoids. However, this comes with the benefit of yielding a tighter estimated set. \hfill $\square$
\end{remark}

\section{Robustness to distributional shifts}

We first define a distance metric between distributions. Due to its optimal transport properties \cite{mohajerin2018}, we use the so-called 1-Wasserstein distance defined as follows:

\begin{definition} \label{def:wass_distance}
Let $(\Delta,\mathcal{G})$ be a measurable space endowed with a metric
$d:\Delta\times\Delta\to\mathbb{R}_{+}$, and let $\mathbb{P},\mathbb{Q}$ be probability measures on
$(\Delta,\mathcal{G})$. The $1$-Wasserstein distance between $\mathbb{P}$ and $\mathbb{Q}$ is defined as:
\[
W(\mathbb{P},\mathbb{Q})\;:=\;\inf_{\pi\in\Pi(\mathbb{P},\mathbb{Q})} \int_{\Delta\times\Delta} d(\delta,\delta')\,\pi(d\delta,d\delta'),
\]
where $\Pi(\mathbb{P},\mathbb{Q})$ is the set of all joint probability distributions with marginals $\mathbb{P}$ and $\mathbb{Q}$, respectively. \hfill $\square$
\end{definition}
\begin{definition} \label{def:ball}
The $1$-Wasserstein ball centered at $\mathbb{P}$ with radius $\tilde{\mu}$ is defined as:
\begin{align}
\mathbb{B}_{\tilde{\mu}}(\mathbb{P})
:=\left\{\,\mathbb{Q}\ \text{probability on }(\Delta,\mathcal{G})\;:\;
W(\mathbb{P},\mathbb{Q})\le \tilde{\mu}\,\right\}. \nonumber \hfill \ \  \square
\end{align} 
\end{definition}
We now define the following out-of-distribution probability metric over the Wasserstein ambiguity set.
\begin{definition} \label{def:ood-violation}
Consider a Wasserstein ball $\mathbb{B}_{\tilde{\mu}}(\mathbb{P})$. Then, the out-of-distribution probability of trajectory exclusion for the collection of sets $\mathcal{R}^\ast_N=\{\mathcal{R}^\ast_{k,N,\rho, A}\}_{k \in \mathcal{T}}$ is defined as:
\begin{align}
\mathbb{V}_{\tilde{\mu}, \mathbb{P}}(\mathcal{R}_N^\ast)
:=\;& \!\!\! \sup_{\tilde{\mathbb{P}} \in \mathbb{B}_{\tilde{\mu}}(\mathbb{P})} \!\!\!
\tilde{\mathbb{P}}\!\left\{
x \in \mathcal{X} : \exists k \in \mathcal{T}
\text{ s.t. } x_k \notin \mathcal{R}^\ast_{k,N,\rho,A}
\right\}.
\nonumber
\end{align}
\end{definition}

Based on Definition \ref{def:ood-violation}, the following result then holds:

\begin{theorem} \label{thm:ood-wass}
Consider system (\ref{state_dynamics}) and the estimated collection of reachable sets $\mathcal{R}_{N}^\ast$  obtained by solving $\hat{P}_{ N, \rho, A}$ under Assumptions~1-4. Consider the ball of distributions  $\mathbb{B}_{\tilde{\mu}}(\mathbb{P})$ of radius $\tilde{\mu}>0$ around $\mathbb{P}$ and the radius of sampled trajectory deviations $R>0$. Let $s^\ast_{A}$ be the adversarial complexity from Definition \ref{def:adv_com} and 
$\overline{\epsilon}(\cdot)$ the violation level obtained from (\ref{eq:eps_poly_nu}) and (\ref{eq:eps_poly_N_nu}). Then, for any $\beta\in(0,1)$, with confidence at least $1-\beta$, it holds that:
\begin{align}
\mathbb{V}_{\tilde{\mu}, \mathbb{P}}(\mathcal{R}_N^\ast) \le\ 
\overline{\epsilon}(s^\ast_{A}) + \frac{\tilde{\mu}}{R}. \nonumber 
\end{align}
\end{theorem}
\emph{Proof}: We follow similar steps as in the proof of Theorem \ref{thm:1}, and then apply Theorem 5 in \cite{Campi_RiskAnalysis_2025}. \hfill $\blacksquare$

\begin{figure*}[t]
    \centering

    \begin{subfigure}{1\textwidth}
        \centering
        \includegraphics[width=\textwidth]{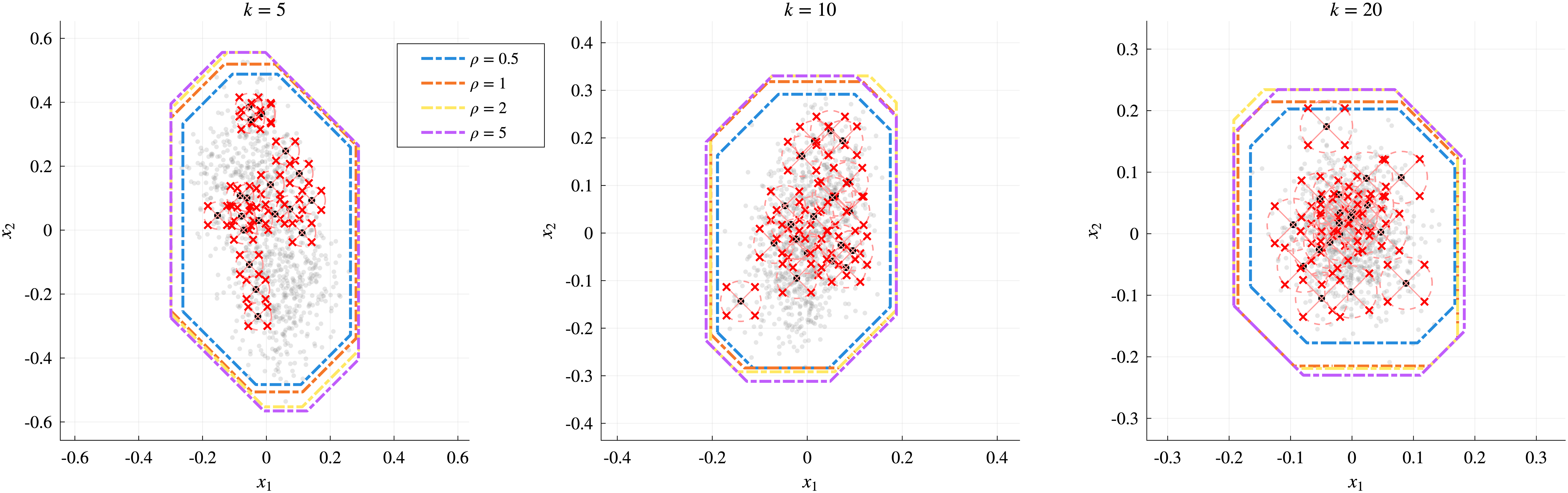}
        \caption{ \footnotesize Adversarially robust reachable set based on zonotopes.}
        \label{fig:reach-zono}
    \end{subfigure}

    \vspace{0.6em}

    \begin{subfigure}{1\textwidth}
        \centering
        \includegraphics[width=\textwidth]{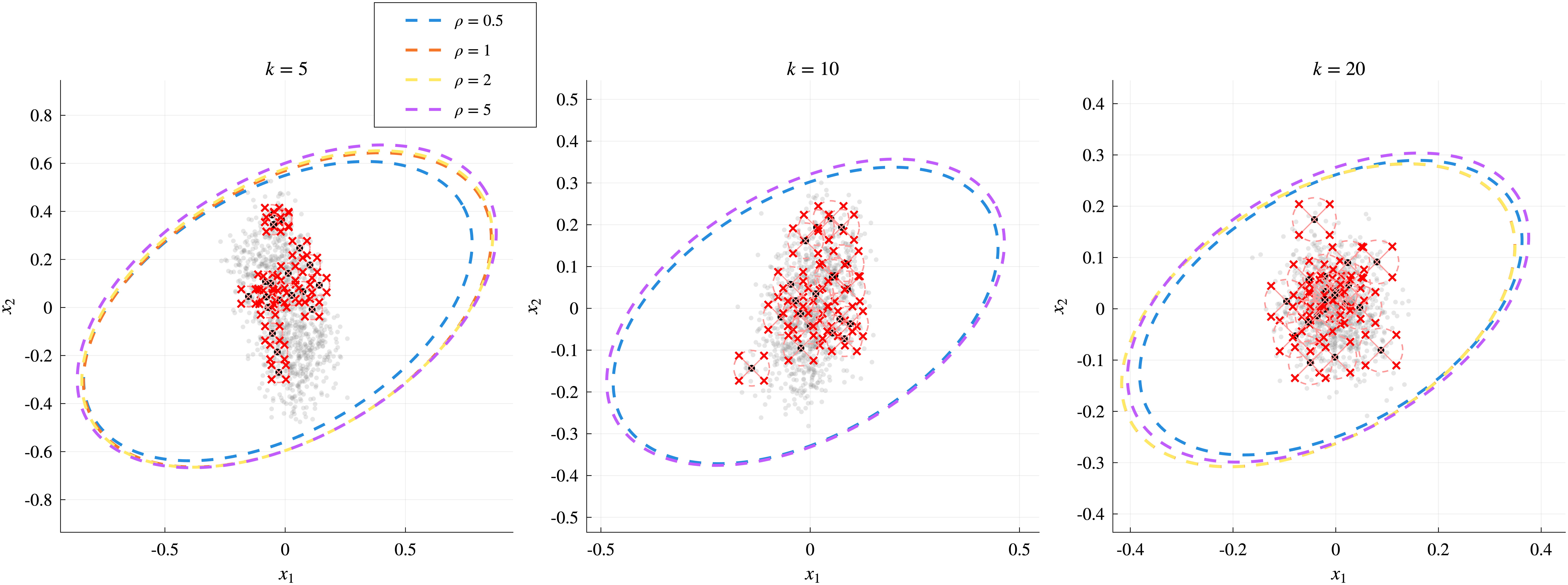}
        \caption{\footnotesize Adversarially robust reachable set based on ellipsoids.}
        \label{fig:reach-elli}
    \end{subfigure}

    \vspace{0.6em}

    \begin{subfigure}{1\textwidth}
        \centering
        \includegraphics[width=1\textwidth]{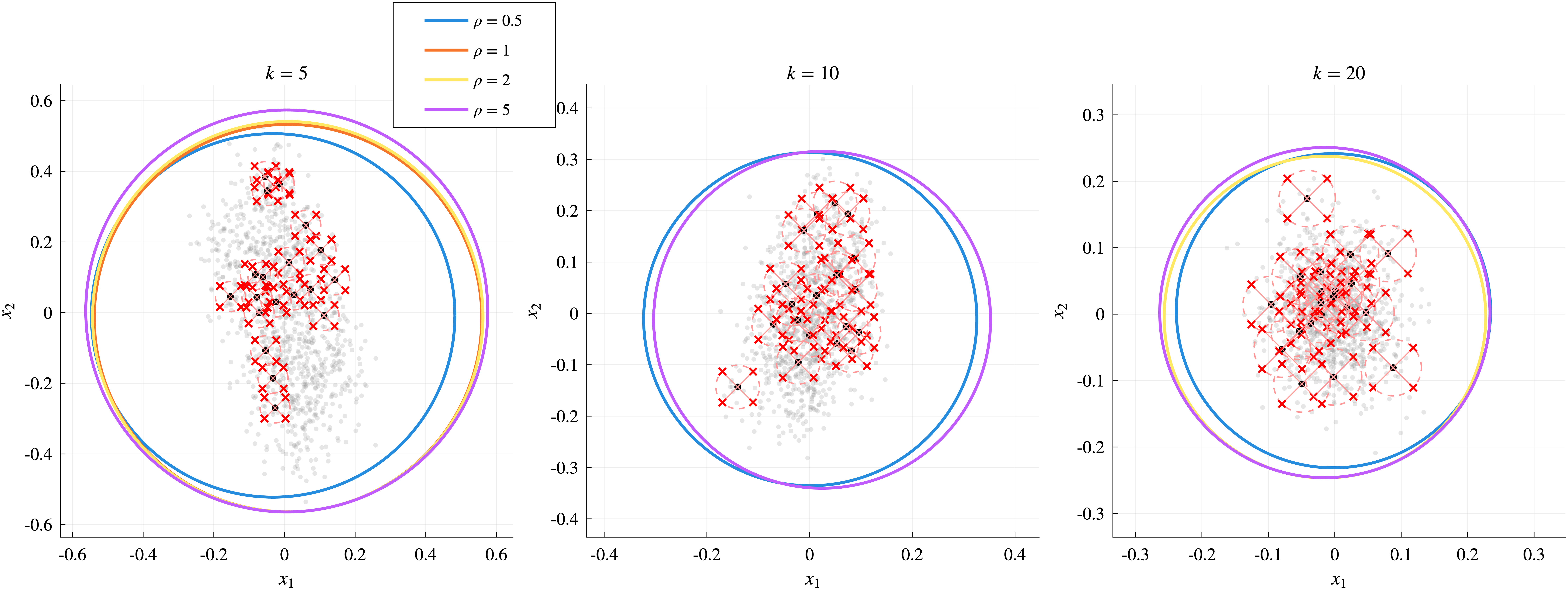}
        \caption{\footnotesize Adversarially robust reachable set based on Euclidean balls.}
        \label{fig:reach-ball}
    \end{subfigure}

    \caption{\footnotesize Adversarially robust data-driven estimation of reachable sets for different geometries, i.e., zonotopes, ellipsoids and Euclidean balls. We consider $T=25$ and illustrate three timestamps at $k\in \{5,10,20\}$.  We learn each set using $N=1000$ different state trajectories, perturbed by the set $\Delta$, thus giving rise to the red cross points which denote adversarial samples. We consider a penalty parameter $\rho \in \{0.5, 1, 2, 5\}$ and illustrate the geometric shapes for these values. Note that as $\rho$ increases, the set increases as well. This is because the relaxation variables of $\hat{P}_{ N, \rho, A}$ are penalized more to satisfy the observed state trajectories. As such, $\rho$ acts as a tuning parameter trading set size for inclusion of future state trajectories.}
    \label{fig:reach-subfigs}
\end{figure*}

Theorem \ref{thm:ood-wass} implies that adversarial robustness against perturbations of state trajectories inherits distributional robustness properties. Note that the guarantees are similar to those of Theorem \ref{thm:1} with an additional term that depends on the ratio between the Wasserstein radius and the adversarial radius. 

\begin{figure}[t]
    \centering
\includegraphics[width=0.49\textwidth,trim=0 0 0 0cm, clip]{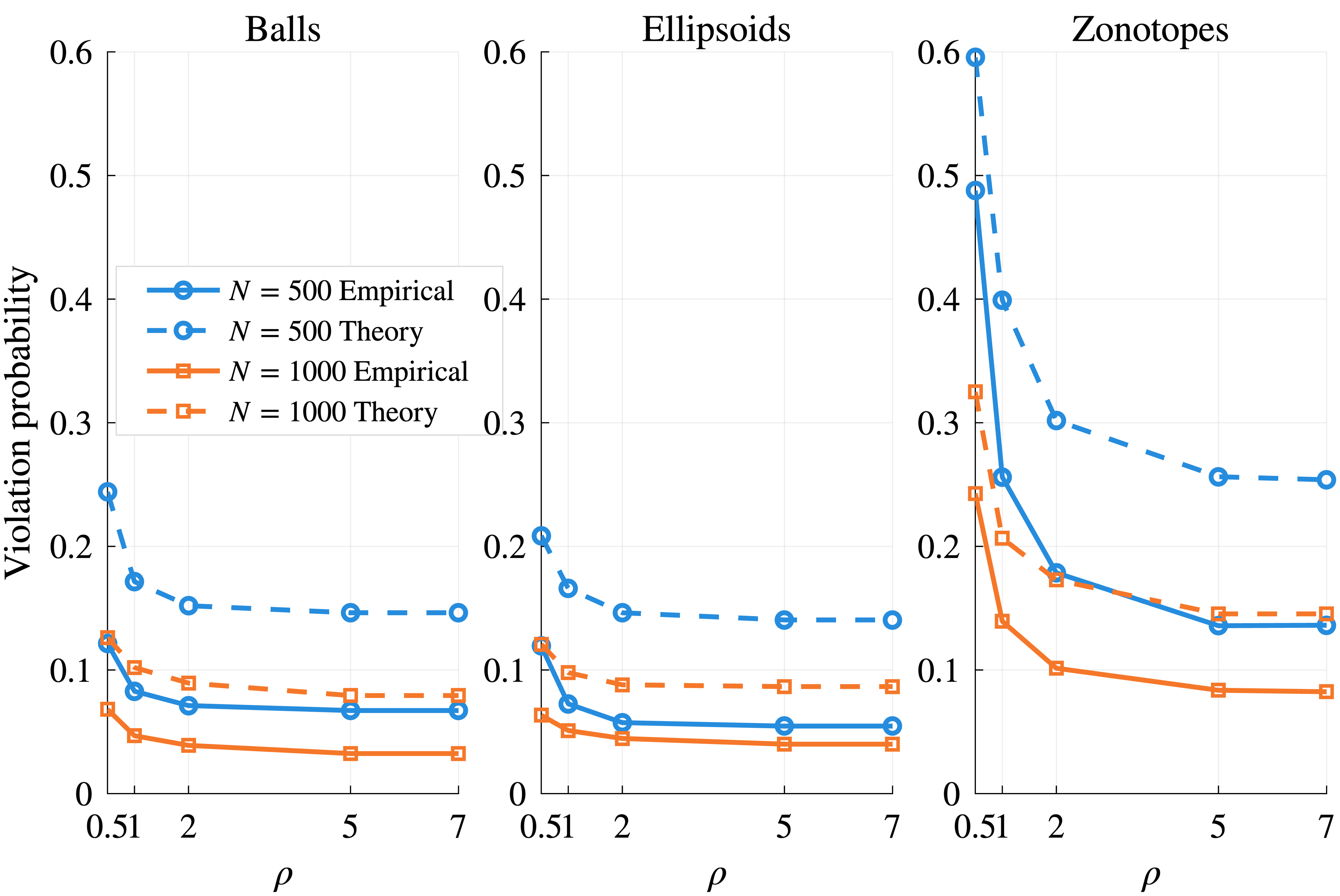}
    \caption{\footnotesize Adversarial theoretical trajectory violation level (dashed lines) vs its empirical counterpart (solid lines) for different geometric shapes. As the number of samples $N$ increases, the theoretical and empirical probabilities improve significantly. Note that zonotopes require a higher number of sampled trajectories to provide the same guarantees with ellipsoids and balls, though at the benefit of tighter zonotopic regions.} 
    \label{fig:adv_theor_vs_empir}
\end{figure}

\section{Numerical Simulations}

\subsection{Adversarially robust reachable set estimation} \label{sub:numerical_example1}

We consider the following stochastic dynamical system:
\begin{equation} \label{eq:example}
x_{k+1} = A x_k + B \phi(C x_k) + w_k,
\end{equation}
where $\phi(z) = a \tanh(z)$ with $a=-0.9$, and the matrices:
\begin{align}
A=\begin{bmatrix}
0.95 & 0.10\\
-0.20 & 0.85
\end{bmatrix},
B=\begin{bmatrix}
0.18\\
0.06
\end{bmatrix},
C=\begin{bmatrix}
1 & 0
\end{bmatrix}. \nonumber
\end{align}

We consider a horizon of length $T=25$. We draw initial conditions $x_0$ from the box $[-0.6, 0.6]\times[-0.45, 0.45]$ and the disturbance $w_k$ from a uniform probability distribution with support $[-0.05, 0.05]^2$. By feeding $w_k^{(i)}$ and $x_0^{(i)}$ through the dynamics, we obtain the sample trajectory $x^{(i)}=\{x_k^{(i)}\}_{k \in \mathcal{T}}.$
Using $N$ such trajectories, we obtain the multi-sample $\{x^{(i)}\}_{i \in \mathcal{N}}$ and we model the presence of possible data perturbations by considering a perturbation to each sampled trajectory, i.e., $
\tilde{x}_k^{(i)} = x_k^{(i)} + \delta$,
where
$\Delta = \{\delta \in \mathbb{R}^2 : \|\delta\|_\infty \le \gamma\}$, with $\gamma=0.03$.
Since $\Delta$ is a rectangle, we can use the vertices to define the set $\mathcal{M}$ in $\hat{P}_{N,\rho, A}$. We now define the empirical adversarial probability of trajectory exclusion as follows:
\begin{align}
\hat{V}_{\mathrm{adv}}
= \frac{1}{N_{\mathrm{test}}} \sum_{m=1}^{N_{\mathrm{test}}}
\Ind[ \max_{k \in \mathcal{T}}\max_{j \in \mathcal{M}}g_k(x_k^{(m,j)}, \theta^\ast_{k, N, \rho, A}) > 0], \nonumber 
\end{align}
where $N_{test}$ is the number of test samples and $\mathcal{M}$ is the index set denoting the vertices of the perturbation set $\Delta$. Leveraging the results from Lemma \ref{lem:2} we solve problem $\hat{P}_{ N, \rho, A}$ for system \eqref{eq:example} for three different geometric shapes, namely, Euclidean balls, ellipsoids and zonotopes. In Figure \ref{fig:reach-subfigs} we set a horizon of $T=25$ and illustrate the learned sets at $k \in \{5,10,20\}$, trained from $N=1000$ state trajectories, including adversarially perturbed samples from $\Delta$ (red crosses). Note that perturbations are considered for each sample. We chose to depict only 20 of them in Figure \ref{fig:reach-subfigs} for illustrative purposes. Shapes are shown for $\rho\in\{0.5,1,2,5\}$. Larger $\rho$ penalizes the relaxation variables more, yielding larger sets that better cover the data. As such, $\rho$ trades set size against out-of-sample performance. Figure \ref{fig:adv_theor_vs_empir} shows the adversarial theoretical and the adversarial empirical probability of violation for different number of samples $N \in \{500,1000\}$ and penalty parameters $\rho \in \{0.5, 1 , 2, 5, 7\}$. Note that the empirical probability of trajectory violation is always below the corresponding theoretical bound.

\subsection{Distributional robustness through adversarial training}
In this section we generate $5$ different experiments using $N=1000$ different trajectory samples per experiment. We consider the same system \eqref{eq:example} as in Section \ref{sub:numerical_example1} with the initial condition $x_0$ now drawn from a Gaussian distribution $\mathcal{N}(\mu_{x_0}, \text{diag}(\sigma_{x_0}) )$ where $\mu_{x_0}=(0, 0)^\top$, and $\sigma_{x_0}=(0.3, 0.225)^\top$ and  $w_k \sim \mathcal{N}(\mu_w, \sigma_w^2 I_2)$, with $\mu_w=(0, 0)^\top$ and $\sigma_w=0.0167$. We now consider $N_{test}=3000$ state trajectories drawn from the perturbed distributions  $\mathcal{N}(\hat{\mu}_{x_0}, \text{diag}(\hat{\sigma}_{x_0}) )$ where $\hat{\mu}_{x_0}=(0.01, -0.01)^\top$, and $\hat{\sigma}_{x_0}=(0.315, 0.23625)^\top$ and  $\hat{w}_k \sim \mathcal{N}(\hat{\mu}_w, \hat{\sigma}_w^2 I_2)$, with $\hat{\mu}_w=(0.002, -0.002)^\top$ and $\hat{\sigma}_w=0.0175$. With these samples, we calculate the out-of-distribution empirical probability of violation and compare it with the theoretical guarantees in Theorem \ref{thm:ood-wass}, where $\tilde{\mu}=0.0243$ is an upper bound that guarantees that the test distributions lie within this Wasserstein ball. The results are depicted in Figure \ref{fig:odd_bounds}. Note that the empirical out-of-distribution probability of trajectory violation is always below the corresponding theoretical trajectory violation bound for all considered geometries. It is important to note that the distributional robustness is a property of adversarial training. Thus, for significantly different distributions, the second term in the bound of Theorem \ref{thm:ood-wass} can increase significantly unless $R$ is tuned differently. Future work will focus on establishing a distributionally robust methodology incorporated in the scenario methodology such that the distributional robustness is explicitly guaranteed. Figure  \ref{fig:relative_size} shows the relative cumulative reachable set size for each geometry, where $\mathrm{Size}_{\mathrm{ball}}(\rho):=\sum_{k\in\mathcal{T}} r_k(\rho)$, $\mathrm{Size}_{\mathrm{ell}}(\rho):=\sum_{k\in\mathcal{T}} s_k(\rho)$, and $\mathrm{Size}_{\mathrm{zono}}(\rho):=\sum_{k\in\mathcal{T}} \mathbf{1}_m^\top a_k(\rho)$. Each quantity is normalized by its corresponding value at $\rho=\rho_0=0.5$. Note that as $\rho$ increases, all geometries interestingly exhibit a similar relative increase in cumulative size.
\begin{figure}[t]
    \centering
    \includegraphics[width=0.5\textwidth,trim=0 0 0 0cm, clip]{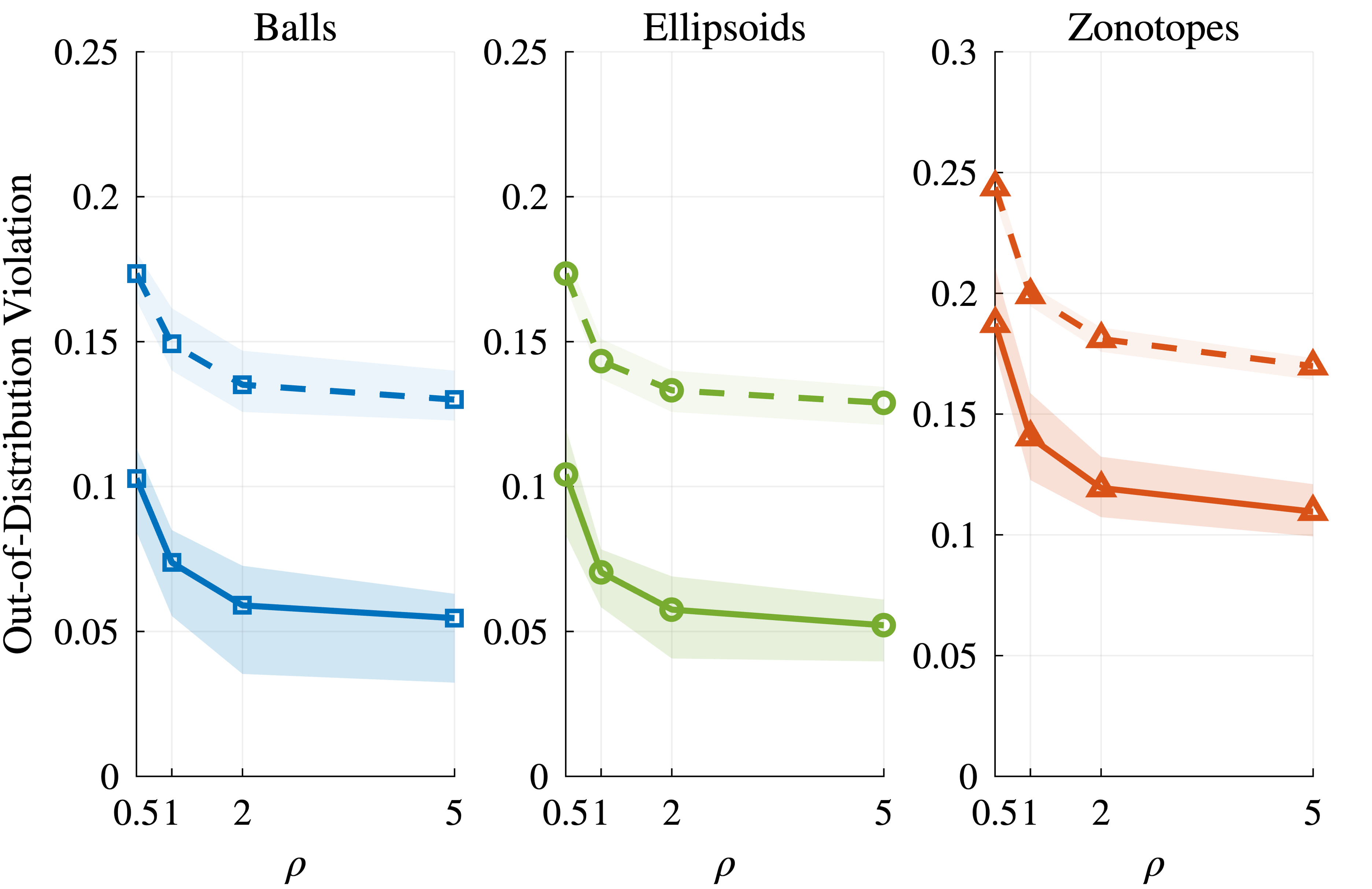}
    \caption{\footnotesize Out-of-distribution theoretical trajectory violation level (dashed lines) and out-of-distribution empirical trajectory violation level for reachable sets of different geometries.} 
    \label{fig:odd_bounds}
\end{figure}

\begin{figure}[t]
    \centering
    \includegraphics[width=0.5\textwidth,trim=0 0 0 0cm, clip]{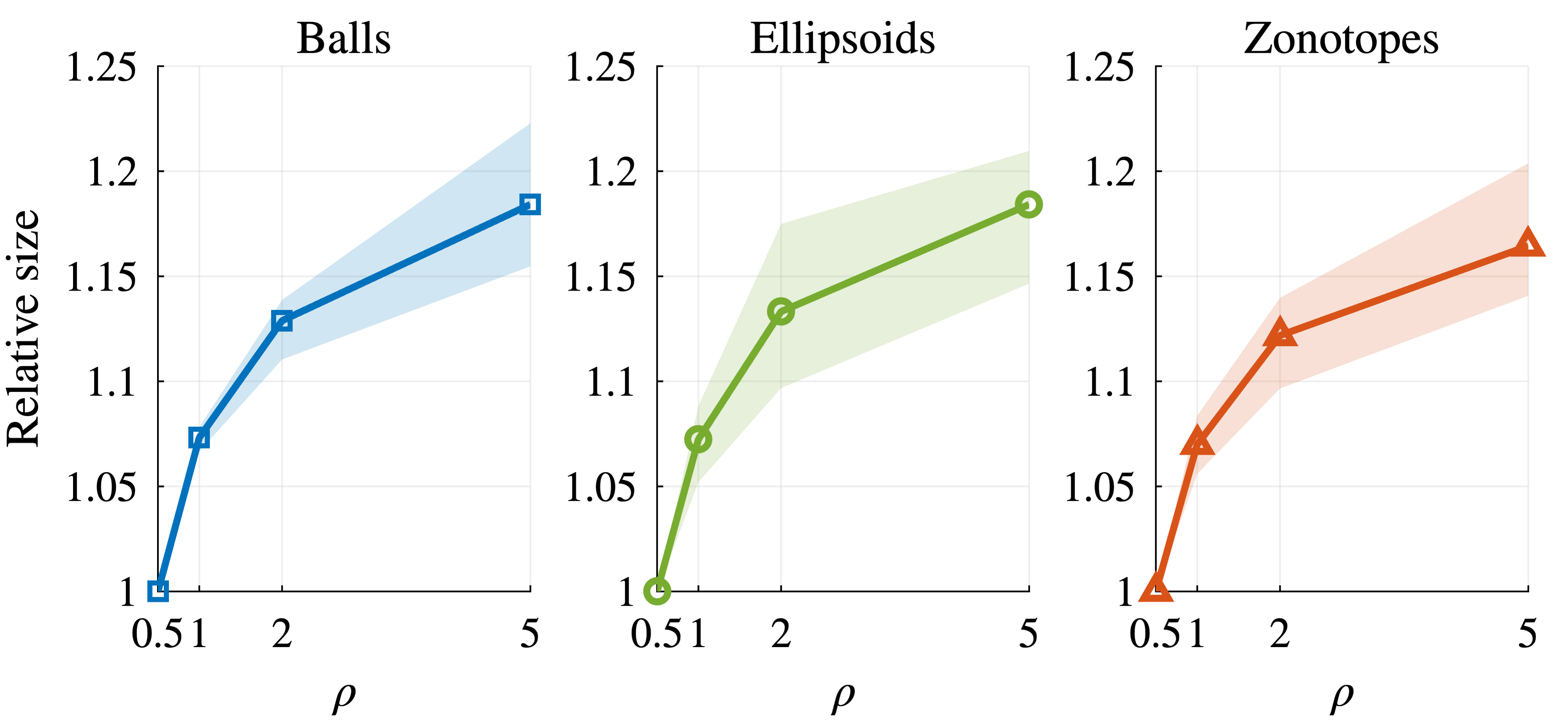}
    \caption{\footnotesize Relative cumulative reachable set size over time for each geometry. Interestingly, as $\rho$ increases, all shapes exhibit similar relative increase in cumulative size.}
    \label{fig:relative_size}
\end{figure}

\section{Conclusion}

We propose a methodology for reachable set estimation, based on recent results in the scenario approach, that learns convex reachable sets directly from trajectory data, explicitly balancing set size and generalization properties of the set through a penalty hyperparameter. Furthermore, we obtain a posteriori robustness certificates to bounded adversarial perturbations and Wasserstein distribution shifts. A key next step is to make distributional robustness explicit by embedding a Wasserstein ambiguity set into the training procedure, rather than relying on robustness inherited from adversarial training.  

\section{Appendix}

\emph{Proof of Theorem 1}: The following equalities hold: 
\begin{align}
&\mathbb{V}_A(\mathcal{R}^\ast_N)
:= \mathbb{P}\Bigl\{x : \exists\, \tilde{x}\in A_{x},\ \exists\, k\in \mathcal{T} \text{ s.t. } \tilde{x}_k \notin \mathcal{R}^\ast_{k,N,\rho,A}\Bigr\} \nonumber \\
&= \mathbb{P}\Bigl\{x : \exists\, \tilde{x}\in A_{x}, \text{such that } \max_{k \in \mathcal{T}}g_k(\tilde{x}_k, \theta^\ast_{k,N, \rho, A}) > 0\} \nonumber \\
&=\mathbb{P}\Bigl\{x : \exists\, \tilde{x}\in A_{x}, \text{such that } g(\tilde{x}, \theta^\ast_{N, \rho, A}) > 0\}.  \nonumber
\end{align}
Application of Theorem 1 in \cite{Campi_RiskAnalysis_2025} concludes the proof.  \hfill $\blacksquare$

\emph{Proof of Lemma 1}: i) The proof follows by setting $\theta=\big((c_k)_{k\in\mathcal{T}};\,(r_k)_{k\in\mathcal{T}}\big)$, $g(x^{(i,j)}, \theta)=\max_{k \in \mathcal{T}} \|x^{(i,j)}_k-c_k\|_p -r_k $ in $\hat{P}_{ N, \rho, A}$ and taking the volume  of the $p$-normed ball $\text{Vol}(\mathcal{R}_k)=\alpha_{n_x,p}r_k^{n_x}$.
ii) The proof follows by setting $g(x^{(i,j)}, \theta)=\max_{k \in \mathcal{T}} \|C_kx^{(i,j)}_k+b_k\|_2-1$ in $\hat{P}_{ N, \rho, A}$ and taking $\text{Vol}(\mathcal{R}_k)=\dfrac{\alpha_{n_x,2}}{\det(C_k)}$.

\emph{Proof of Lemma 2}: i) The proof is similar to the corresponding proof of Lemma \ref{lem:reformulations}. 
ii) If $H_k$ is fixed for all $k$,  setting $g(x, \theta)= \max_{k \in \mathcal{T}}\|H_k(x^{(i,j)}_k-c_k)\|_2-s_k$ yields the desired reformulation. iii)  For each time \(k\), fix a generator matrix \(G_k\in\mathbb{R}^{n_x\times m}\) .  
We model the reachable set as the zonotope $\mathcal{R}_k(c_k,a_k)=\{x\in\mathbb{R}^{n_x}:\exists \zeta\in\mathbb{R}^m,\ x=c_k+G_k\zeta,\ |\zeta| \le a_k\}$,
with center \(c_k\in\mathbb{R}^{n_x}\) and half-widths \(a_k\in\mathbb{R}_+^m\).
Introducing the auxiliary variables \(\zeta_{i,k,j}\in\mathbb{R}^m\), then the above set is equivalent to satisfying the constraints $c_k + G_k \zeta_{i,k,j} = x_k^{(i,j)}$ and
$-\,(a_k+\xi_i \mathbf{1}_m)\le \zeta_{i,k,j}\le a_k+\xi_i \mathbf{1}_m$ componentwise. \hfill $\blacksquare$
\bibliographystyle{IEEEtran}
\bibliography{biblio}
\end{document}